\documentclass{amsart}
\usepackage{tikz-cd}
\usepackage{tikz}
\usepackage[utf8]{inputenc}
\usetikzlibrary{matrix,arrows.meta}
\usepackage{amsmath}

\usetikzlibrary{calc}
\usepackage{caption}
\usepackage{longtable}
\usepackage{graphicx}
\usepackage{subcaption}
\usepackage{epsfig}
\usepackage{color}
\usepackage{array}
\usepackage{float}
\usepackage{amsmath}
\usepackage{amsthm}
\usepackage{amssymb}
\usepackage{enumerate}
\usepackage{amscd}
\usepackage{color}
\usepackage{relsize}
\usepackage[bookmarksnumbered,bookmarksopen,pdfusetitle,unicode]{hyperref}
\usepackage[all,cmtip]{xy}
\usepackage{ctable}
\theoremstyle{plain}
\newtheorem{thm}{Theorem}[section]
\newtheorem{cor}[thm]{Corollary}
\newtheorem{lem}[thm]{Lemma}

\newtheorem{conj}[thm]{Construction}

\theoremstyle{definition}

\theoremstyle{remark}
\newtheorem{rem}[thm]{Remark}
\newtheorem{defn}[thm]{Definition}

\newtheorem{notn}[thm]{Notation}

\usepackage{fancyhdr}
\makeatletter
\let\runauthor\@author
\let\runtitle\@title
\makeatother

\makeatletter
\@namedef{subjclassname@2020}{\textup{2020} Mathematics Subject Classification}
\makeatother

\begin{document}
	
	\title{Gröbner Fans and Minimal Embedded Toric Resolutions of Rational  Double Point Singularities} 
	
	\author{B\"u{\c s}ra KARADENİZ {\c S}EN}
	
	
	

	\begin{abstract}

In \cite{hc}, the authors give minimal embedded toric resolutions of $ADE$-singularities in $\mathbb{C}^3$ by constructing regular refinements of their dual Newton polyhedrons with the elements of their embedded valuation sets derived from the jet schemes constructed in \cite{hussein}. On the other hand in  \cite{AGS} and \cite{ar-hu}, the authors represent the Gröbner fan of a Newton non-degenerate variety and prove that a regular refinement of the Gröbner fan of such a singularity yields an embedded toric resolution. In this paper, we reconstruct embedded toric resolutions of $ADE$-singularities: We give the explicit constructions of  their Gröbner fans and refine them using the concept of profile.

	\end{abstract}
	\maketitle
 
\section{Introduction}

\noindent Let $X\subset \mathbb{C}^n$ be a variety and let $Sing(X)$ be its singular locus. A resolution of singularities of $X$ is a proper birational morphism $\pi: \tilde{X}\to X$ such that $\tilde{X}$ is a non-singular variety and $\tilde{X} \backslash \pi^{-1}(Sing(X))\cong X\backslash Sing(X)$. By  \cite{hironaka}, it is known that $X$ admits a resolution. 

\noindent J. Nash introduced in \cite{Nash} the arc space $J_{\infty}(X):=\{\gamma: Spec \mathbb{C}[[t]]\to X\}$ of $X$ and he showed that the number of irreducible components of $J_{\infty}(X)$ passing through $Sing(X)$ is at most the number of essential irreducible components of the exceptional fibre of a resolution. The conjecture the equality holds is called Nash Problem. In dimension $2$, means for surfaces, the Nash Problem was topologically solved with positive answer in \cite{bobadilla} by J. Fernandez Bobadilla and M. Pe Pereira. Later, it was algebraically solved in \cite{dedo} by T. De Fernex and R. De Campo. In higher dimension, the Nash Problem had positive answer for some special cases (see \cite{Fernex, IK, plenat-pampu} for more detail). On the other hand, the $m$-th jet scheme $J_{m}(X):=\{\gamma_m: Spec \frac{\mathbb{C}[[t]]}{\langle t^{m+1} \rangle}\to X\}$ of $X$ is the set of truncated arcs. The arc space $J_{\infty}(X)$ of $X$ may be seen as the limit of the jet scheme $J_m(X)$ of $X$ \cite{Fernex2}. After the studies over Nash Problem, in $2013$, M. Lejeune-Jalabert, A. Reguera and H. Mourtada introduced in the Inverse Nash Problem which is \textquotesingle Given the jet scheme $J_m(X)$ of $X$, can we construct a resolution of singularities of $X$?\textquotesingle. In \cite{hc}, it was answered by C. Plénat and H. Mourtada for a special class of singularities which is called rational double point singularities ($ADE$-singularities). They constructed minimal embedded toric resolutions of them by using their jet schemes given in \cite{hussein} where the authors gave the jet schemes of $ADE$-singularities and provided the embedded valuation sets which are the set of vectors coming from the some special irreducible components of the jet schemes. Moreover, in \cite{hussein}, one can find a relation between the number of irreducible components of jet schemes passing through the singular locus and the number of exceptional curves on the minimal resolution of the singularities. There is a one to one correspondence between them.

\noindent Furthermore in \cite{koreda}, Y. Koreeda focused on the relation between the jet schemes of $ADE$-singularites and their minimal resolution graphs. He proved that a graph obtained by the jet schemes is isomorphic to the minimal resolution graph for $A_n$-singularities and $D_4$-singularity \cite{koreda}. For the other members of $ADE$-singularities, showing this is still an open question. Despite this, it is a significant progress since it is obvious that there is a relation between the irreducible components of jet schemes and the minimal resolution graph of the singularities.

\noindent After all these studies, in this paper we give a different point of view of the resolution of singularities for $ADE$-singularities. Mainly we focus on the resolutions by constructing regular refinements of the Gröbner fans of them.

 \noindent In $1970$'s, Newton non-degenerate hypersurface singularities were introduced by A. G. Khovanskii and A. G. Kouchnirenko. There is a guarantee on a Newton non-degenerate hypersurface singularity has a toric resolution \cite{khovanskii, kouch, oka, Varc}. As we mentioned before, in \cite{hc}, the authors constructed resolutions for $ADE$-singularities which are Newton non-degenerate hypersurface singularities by giving regular refinements of their dual Newton polyhedrons. On the other hand, in \cite{ar-hu}, the authors proved that the Gröbner fan is a polyhedral fan and a variety defined by Newton non-degenerate ideal on a toric variety with any characteristic of the base field admits an embedded toric resolution of singularities by constructing regular refinements of their Gröbner fans.
\noindent The motivation of this article is the following theorem:
\begin{thm}
\cite{AGS, ar-hu} Let $I\subset \mathbb{C}[x_1,x_2,\ldots,x_n]$ be a Newton non-degenerate ideal and let $X=V(I)\subset X_{C}$ where $X_{C}$ is the toric variety associated to the cone $C$. Let $\Sigma$ be a regular refinement of $C$ which is compatible with Gröbner fan of $X$. Then the associated toric morphism $\pi_{\Sigma}: X_{\Sigma}\to X_{C}$ is a proper birational morphism and the irreducible components of the total transform $\pi_{\Sigma}^{-1}$ are smooth and meet transversally.      
\end{thm}

 \noindent More precisely, here following \cite{AGS} and \cite{ar-hu} we give the explicit constructions of Gröbner fans of $ADE$-singularities. Then we provide minimal embedded toric resolutions of them by constructing regular refinements of their Gröbner fans. To do this, we look at the profile of each maximal dimensional Gröbner cones in their Gröbner fans and we use the vectors living inside the profile. The concept of profile was introduced in \cite{cg}. Then in \cite{bcm} the authors extended the definition to non-simplicial cones and used it to show the minimality of the resolutions for a class of singularities which is called rational triple point singularities ($RTP$-singularities). Thus the idea of using the profile to find suitable vectors for a regular refinement comes from the studies in \cite{bcm}.
 
\noindent This study is inspired by the works of C. Plénat and H. Mourtada. The reader will find the reconstructed version of the work from \cite{hc} using the Gröbner fan. What distinguishes this work is that we construct minimal embedded toric resolutions of $ADE$-singularities, a fundamental class of singularities, via their Gröbner fans, rather than the classical method of utilizing their dual Newton polyhedrons. Also the reader will find the brief reminder of the works in \cite{koreda, hussein}. All these come together, this survey provides a broad perspective on the resolutions of $ADE$-singularities.

\noindent This article is organized as follows: We start with the profile of a cone in Section $2$. The reader can be also find some basic definitions in this section. The Section $3$ is devoted to explain Newton non-degeneracy and the notion of Gröbner fan. Then in Section $4$, we give minimal embedded toric resolutions of $ADE$-singularities. In each sub-section, one can find the detailed computations for each of the sub-classes. 

\section{Profile of a Cone}
\noindent Let $v_1,v_2,\ldots,v_r$ be the vectors in $\mathbb{Z}^n$. A rational polyhedral cone in $\mathbb{R}^n$ generated by $v_1,v_2,\ldots,v_r$ is the set $$C:=\langle v_1,v_2,\ldots,v_r \rangle =\{ \lambda_1 v_1+\lambda_2 v_2+\ldots+ \lambda_r v_r \mid \lambda_i \in \mathbb{R}_{\geq 0}, i=1,2,\ldots,r \}$$ A rational polyhedral cone is called strongly convex if it does not contain any linear subspace different from $\{ \textbf{0} \}$. A vector is said to be primitive if all its coordinates are relatively prime. From now on, we assume that all cones $C$ are strongly convex with primitive generators. 

\noindent  Let us associate with $C=\langle v_1,v_2,\ldots,v_r \rangle \subset \mathbb{R}^n$ the matrix $M_{C}:=[v_1 \ v_2 \ldots v_r]$ where the colums of $M_{C}$ are $v_{i}$'s. Consider the $r\times r$ minors $M_r$ of $M_{C}$. The determinant of $C$ is $$det(C):=gcd(det(M_r))$$
\begin{defn}
A cone $C \subset \mathbb{R}^n$ is called regular if $det(C)=\pm 1$. Otherwise $C$ is called non-regular.  
\end{defn}

\noindent \cite{cox} A collection of cones $\Sigma\subset \mathbb{R}^n$ is called a fan if 

$i)$ Each face of a cone in $\Sigma$ is a cone in $\Sigma$.

$ii)$ The intersection of any two cones $C_1, C_2 \in \Sigma$ is a face for both $C_1$ and $C_2$.

\noindent A fan $\Sigma\subset \mathbb{R}^n$ is called regular if each cone in $\Sigma$ is regular. 
\vskip.3cm
\noindent If $C\subset \mathbb{R}^n$ is a non-regular cone then one can obtain a regular sub-cones from $C$ by adding vectors. This process is called regular refinement of $C$. Each non-regular cone has a regular refinement \cite{fulton}. 

 \noindent To construct a regular refinement of a non-regular cone $C$, we use the concept of profile introduced in \cite{cg}. The profile of a cone $C$ is a specific bounded region. Here we count the vectors with integer entries living inside the profile of $C$. Specifically, we look at the vectors from the intersection of the profile of $C$ and $\mathbb{Z}^n$.
\begin{defn}
\cite{cg} The primitive vector over $1$-dimensional face of $C$ is called extremal vector.  
\end{defn}
\noindent For $C=\langle v_1,v_2,\ldots, v_r \rangle \subset \mathbb{R}^n$, if $n=r$ then we say the cone $C$ is simplical. In \cite{cg}, the authors give the definition of profile for a simplicial cone. In \cite{bcm}, the authors extend the definition of profile to non-simplicial cones. Here our purpose is to use vectors inside the profiles of the maximal dimensional Gröbner cones of $X$ to construct an embedded toric resolution where $X$ corresponds to each of the members of $ADE$-singularities.

\begin{defn}
\cite{bcm} The profile $p_{C}$ of a cone $C=\langle v_1,v_2,\ldots,v_r \rangle \subset \mathbb{R}^n$ is the convex hull such that its extremal vectors are exactly $v_1,v_2,\ldots,v_r$.      
\end{defn}

\noindent The extremal vectors of $C$ are on the hyperplanes which are called the boundary of $p_{C}$. The case $C$ is simplicial, all extremal  vectors are on a unique hyperplane. The case $C$ is non-simplicial, the boundary is the union of hyperplanes. It may happen that though $C$ is a non-simplicial cone, all extremal vectors are on a unique hyperplane as in the case of rational double point singularities.

\section{Newton Non-Degeneracy and Gröbner Fan}
\noindent Let $X$ be a hypersurface defined by $$f(x_1,x_2,\ldots,x_n)=\sum_{\pmb{\alpha} \in \mathbb{Z}^n, c_{\pmb{\alpha}}\in\mathbb{C}} c_{\pmb{\alpha}}x_{1}^{\alpha_1}x_{2}^{\alpha_2}\cdot \cdot \cdot x_{n}^{\alpha_n} $$ where $\pmb{\alpha}=(\alpha_1,\alpha_2,\ldots,\alpha_n)$. The support set of $f$ is $$supp(f):=\{\pmb{\alpha}=(\alpha_1,\alpha_2,\ldots,\alpha_n)\in \mathbb{Z}^n \mid c_{\pmb{\alpha}}\neq 0  \}$$ The Newton polyhedron of $f$ is the convex hull given as  $$NP(f):=conv\{ \pmb{\alpha}+(\mathbb{R}_{\geq 0})^n \mid \pmb{\alpha}\in supp(f) \}$$ Let $F$ be a face of $NP(f)$. The restriction of $f$ to the set $F\subset \mathbb{Z}^n$ is $$f|_{F}:=\sum_{\pmb{\alpha}\in supp(f)\cap F\subset \mathbb{Z}^n} c_{\pmb{\alpha}} \pmb{x}^{\pmb{\alpha}}$$

\begin{defn}
\cite{AGS, oka} A hypersurface singularity $X:=V(f)\subset \mathbb{C}^n$ is called Newton non-degenerate if for each face $F$ of $NP(f)$, the hypersurface $V(f|_{F})$ has no singularities in $(\mathbb{C}^{*})^n$. 
\end{defn}

\noindent Newton non-degenerate hypersurface singularities have been studied remarkably since we know that one may construct embedded toric resolutions of singularities from their dual Newton polyhedrons \cite{khovanskii, oka, Varc}. Here we also study a special class of  Newton non-degenerate hypersurface singularities but in a different way: We construct embedded toric resolutions of them from their Gröbner fans. Let us introduce the notion of Gröbner fan:

\noindent Let $f(x_1,x_2,\ldots,x_n)$ be a polynomial in $\mathbb{C}^n$. Take $\pmb{v}=(v_1,v_2,\ldots,v_n)\in (\mathbb{R}_{\geq0})^{n}$. The $\pmb{v}$-order of $f$ is a number given as $$o_{\pmb{v}}(f):=min \{ v_1\alpha_1+v_2\alpha_2+\ldots+v_n\alpha_n  \mid (\alpha_1,\alpha_2,\ldots,\alpha_n)\in supp(f)\}$$ The polynomial $$In_{\pmb{v}}(f):=\sum_{\{\pmb{\alpha}\in supp(f)\mid \sum_{i=1}^{n}v_i \alpha_i =o_{\pmb{v}}(f)\}}c_{\pmb{\alpha}}x_{1}^{\alpha_1}x_{2}^{\alpha_2}\cdot \cdot \cdot x_{n}^{\alpha_n}$$ is called the $\pmb{v}$-initial form of $f$. We define an equivalence relation as $$\pmb{v}\sim \pmb{w} \iff In_{\pmb{v}}(f)=In_{\pmb{w}}(f)$$ for $\pmb{w}=(w_1,w_2,\ldots,w_n)\in (\mathbb{R}_{\geq 0})^{n}$. The closure of the set $$C_{\pmb{v}}(f):=\{ \pmb{w}\in (\mathbb{R}_{\geq 0})^{n} \mid In_{\pmb{v}}(f)=In_{\pmb{w}}(f)\}$$ is a polyhedral cone. The union of these cones form a fan \cite{Mora}. We use the fan formed by the intersection of this fan and the first orthant. It is called the Gröbner fan $G(X)$ of $X$ \cite{AGS}. Each cone in $G(X)$ is called Gröbner cone. As a remark, one can compute the Gröbner fan by using a computer program which is called GFAN \cite{jensen}.

\section{Embedded Toric Resolution}
\noindent The rational double point singularities are given by one of the following equations:

$$A_n: f(x,y,z)=xy-z^{n+1},  \ \  n\in \mathbb{N}$$
$$D_n: f(x,y,z)=z^2-x(y^2+x^{n-2}), \ \ n\in \mathbb{N}, n\geq 4$$
$$E_6: f(x,y,z)=z^2+y^3+x^4$$ 
$$E_7: f(x,y,z)=x^2+y^3+yz^3$$
$$E_8: f(x,y,z)=z^2+y^3+x^5$$

\begin{rem}
Each member of $ADE$-singularities is Newton non-degenerate.
\end{rem}
 \noindent Embedded toric resolutions of $ADE$-singularities were studied in \cite{hussein, hc}. Here differently we construct regular refinements of Gröbner fans of each members of them. This leads us their embedded toric resolutions \cite{AGS, ar-hu}. Additionally, these resolutions will be minimal which means that:

\begin{defn}
An embedded toric resolution is called minimal if 

$i)$ The vectors in the regular refinement are irreducible.

$ii)$ The resolution graph obtaining from the refinement does not have any $-1$ curve means the self-intersection of each vertex is different from $-1$ (this gives us the resolution graph is minimal).
\end{defn}

 \noindent Note that the self-intersection of a vertex is computed as: The number $-s\in \mathbb{Z}_{\leq 0}$ with $s.v=\sum_{i=1}^{m}v_i$ is called the self-intersection of the vertex $v$ in the graph where each $v_i$ is adjacent to $v$ in the regular refinement.
\subsection{$A_n$-singularities}
Consider the hypersurface $X=V(f)\subset \mathbb{C}^3$ where $$f(x,y,z)=xy-z^{n+1}, \ \ n\in \mathbb{N}$$ Given $v_1\in  (\mathbb{R}_{\geq 0})^3$, let $In_{{v_1}}(f)=f$. The polyhedral cone associated with $v_1$ is
$$C_{v_1}(f)=\langle (n+1,0,1),(0,n+1,1) \rangle$$ 

\noindent Given $v_2\in  (\mathbb{R}_{\geq 0})^3$, let $In_{{v_2}}(f)=xy$. The polyhedral cone associated with $v_2$ is $$C_{v_2}(f)=\langle (0,0,1),(n+1,0,1),(0,n+1,1) \rangle$$ 

\noindent Given $v_3\in  (\mathbb{R}_{\geq 0})^3$, let $In_{{v_3}}(f)=-z^{n+1}$. The polyhedral cone associated with $v_3$ is $$C_{v_3}(f)=\langle (1,0,0),(0,1,0),(n+1,0,1),(0,n+1,1) \rangle$$ 
The Gröbner fan $G(X)$ of $X$ and its Gröbner cones can be seen as  
\begin{figure}[H]
    \centering
    \includegraphics[width=0.9\linewidth]{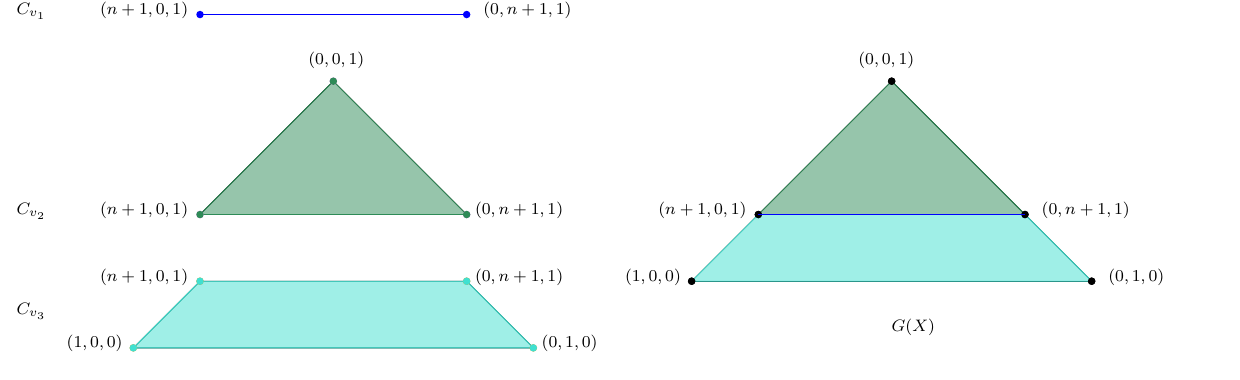}
    \caption{The Gröbner fan and its Gröbner cones for $A_n$-singularities}
\end{figure}

\begin{notn}
We draw the cones/fans by taking the intersection with $x+y+z=1$ to visualize easily.    
\end{notn}

\begin{rem}
The Gröbner fan $G(X)$ of $X$ is same as its dual Newton polyhedron which is given in \cite{hc}. As previously mentioned, the authors construct an embedded toric resolution by means of regular refinement of its dual Newton polyhedron with the vectors coming from its jet schemes. Here we look at the profile to obtain special vectors to do regular refinement of $G(X)$. 
\end{rem}

\noindent For each maximal dimensional cones ($C_{v_{2}}$ and $C_{v_{3}}$) of $G(X)$, the profiles $p_{C_{2}}$ and  $p_{C_{v_{3}}}$ are given by the region shown in the pink color:

\begin{figure}[H]
    \centering
    \includegraphics[width=0.9\linewidth]{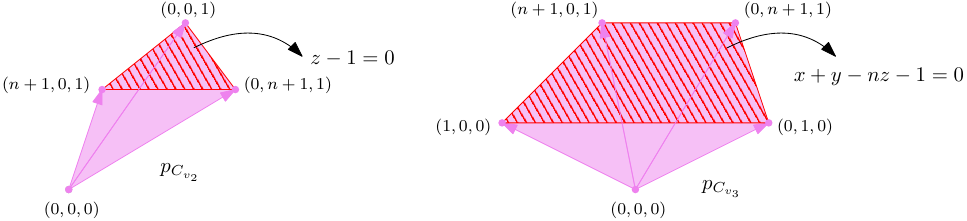}
    \caption{The profiles of maximal dimensional Gröbner cones for $A_n$-singularities}
\end{figure}

\noindent The boundaries of $p_{C_{v_{2}}}$ and  $p_{C_{v_{3}}}$ are $H_1: z-1=0$ and $H_2: x+y-nz-1=0$ respectively. Note that although $C_{v_3}$ is a non-simplicial cone, all its extremal vectors are on a unique hyperplane.
\begin{thm} \label{1}
 $i)$ For $A_n$-singularities, the elements of the set consisting of $\mathbb{Z}^3 \cap p_{C_{v_{i}}}$ give an embedded toric resolution of $X$ where $C_{v_i}$ is a maximal dimensional Gröbner cone in $G(X)$. Moreover these elements are irreducible means they are free over $\mathbb{Z}$.

$ii)$ For $A_n$-singularities, the elements on the skeleton of $G(X)$ give the minimal resolution graph of the singularities.
\end{thm}
\begin{proof}
$i)$ The cones $C_{v_2}$ and $C_{v_3}$ are the maximal dimensional Gröbner cones for $G(X)$. The elements for $\mathbb{Z}^3 \cap p_{C_{v_{2}}}$ are 

$\bullet$ $(n+1,0,1), (n,1,1), \ldots,(0,n+1,1)$

$\bullet$ $(n,0,1), (n-1,1,1), \ldots,(0,n,1)$

$\bullet$ $\vdots$

$\bullet$ $(1,0,1),(0,1,1)$

$\bullet$ $(0,0,1)$

The elements for $\mathbb{Z}^3 \cap p_{C_{v_{3}}}$ are

$\bullet$ $(n+1,0,1), (n,1,1), \ldots,(0,n+1,1)$

$\bullet$ $(1,0,0),(0,1,0)$

\noindent We construct a regular refinement of $G(X)$ with these elements. This refinement is

\begin{figure}[H]
    \centering
    \includegraphics[width=0.9\linewidth]{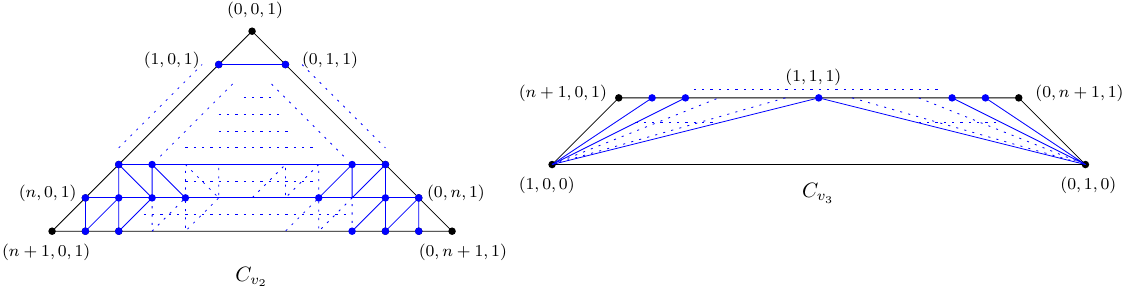}
    \caption{A regular refinement for $A_n$-singularities}
\end{figure}
\noindent Each element in the refinement is irreducible since it cannot be written as a sum of two other elements.

$ii)$ The elements on the skeleton of $G(X)$ gives 
\begin{figure}[H]
    \centering
    \includegraphics[width=0.9\linewidth]{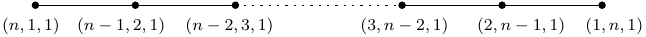}
    \caption{The resolution graph of $A_n$-singularities}
\end{figure}
\noindent The self-intersection of each vertex is $-2$ so it has no $-1$ curves. Hence it is minimal resolution graph.
\end{proof}

\begin{cor}
For $A_n$-singularities, the set of vectors in the intersection of profile of the maximal dimensional Gröbner cones of $G(X)$ and $\mathbb{Z}^3$ are exactly same as the set of embedded valuations of $X$.  
\end{cor}

\noindent This fact can be seen combinatorially by looking at \cite{hussein, hc}.
\begin{cor}
 The embedded toric resolution that we construct is minimal.   
\end{cor}

\begin{proof}
By Theorem \ref{1}($i$), the elements give an embedded toric resolution and these elements are irreducible. By Theorem \ref{1}($ii$), the resolution graph is minimal. Hence the embedded toric resolution we construct is minimal.
\end{proof}

\noindent As we mentioned in the introduction, there are some significant studies about the resolutions of $ADE$-singularities. Here one can find these results: 
\begin{thm}
   \cite{hussein} For $A_n$-singularities, when $n=m$, in $m$-jet scheme over the singular locus there are $n$ irreducible components which is equal to the number of vertices of its minimal resolution graph.
\end{thm}
\noindent  These irreducible components $J_{m}^{i}(X)$, $1\leq i \leq n$, are given by the following ideals in $\mathbb{C}[x_0,\ldots,x_m,y_0,\ldots,y_m,z_0,\ldots,z_m]$ respectively: $$I_{m}^{1}:=\langle x_0,y_0,z_0,x_1,\ldots,x_{(m-1)} \rangle $$  $$I_{m}^{2}:=\langle x_0,y_0,z_0,x_1,\ldots,x_{(m-2)},y_1 \rangle $$ $$\vdots$$ $$I_{m}^{n-1}:=\langle x_0,y_0,z_0,x_1,y_1,y_2\ldots,y_{(m-2)} \rangle $$ $$I_{m}^{n}:=\langle x_0,y_0,z_0,y_1,y_2\ldots,y_{(m-1)} \rangle $$

\noindent In \cite{koreda}, Y. Koreeda give the following construction: When $n=m$, for $A_n$-singularities consider the set of intersection of the irreducible components: $$K:=\{ J_{m}^{i}(X)\cap J_{m}^{j}(X) \mid i\neq j, i,j\in \{1,2,\ldots,n\}\}$$    

\begin{lem}
\cite{koreda} We have the inclusion relation $$I_{m}^{k}\cap I_{m}^{\ell} \subset I_{m}^{i}\cap I_{m}^{j} $$ where $i \leq k \leq \ell \leq j$ which gives $$J_{m}^{i}(X)\cap J_{m}^{j}(X)\subset J_{m}^{k}(X)\cap J_{m}^{\ell}(X)$$    
\end{lem}
\noindent All maximal elemets of the set $K$ with respect to inclusion relation form a subset $E$.
\begin{conj}
\cite{koreda} For some fix $m$, the graph $G$ is constructed with $(V,E)$ as 

$i)$ The irreducible components $J_{m}^{i}(X)$ are the vertices of $G$.

$ii)$ The elements of the set $E$ are the edges of $G$.
\end{conj}
\noindent For $A_n$-singularities, when $n=m$, the vertex set $$\{ J_{m}^{1}(X), J_{m}^{2}(X),\ldots,J_{m}^{n}(X) \}$$ and the edges set is $$\{ J_{m}^{i}(X)\cap J_{m}^{j}(X) \mid j=i+1 \}$$ The graph obtained by the construction is

\begin{figure}[H]
    \centering
    \includegraphics[width=0.7\linewidth]{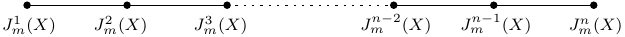}
    \caption{The graph of $A_n$-singularities obtaining from the construction}
\end{figure}

\noindent Note that each $J_{m}^{i}(X)$ correspond to the vector $(n-i+1,i,1)$.
\begin{cor}
\cite{koreda} For $A_n$-singularities, the graph obtained by the construction is isomorphic to its minimal resolution graph.    
\end{cor}
\subsection{$D_n$-singularities}
Consider the hypersurface $X=V(f)\subset \mathbb{C}^3$ where $$f(x,y,z)=z^2-xy^2-x^{n-1}, \ \ n\in \mathbb{N}, n\geq 4$$ The Gröbner cones are $$C_{v_1}(f)=\langle (2,n-2,n-1) \rangle$$
$$C_{v_2}(f)=\langle (2,0,1),(2,n-2,n-1) \rangle$$
$$C_{v_3}(f)=\langle (0,0,1),(2,n-2,n-1) \rangle$$
$$C_{v_4}(f)=\langle (0,1,0),(2,n-2,n-1) \rangle$$
$$C_{v_5}(f)=\langle (2,0,1),(0,0,1),(2,n-2,n-1) \rangle$$
$$C_{v_6}(f)=\langle (0,1,0),(0,0,1),(2,n-2,n-1) \rangle$$
$$C_{v_7}(f)=\langle (1,0,0),(0,1,0),(2,0,1),(2,n-2,n-1) \rangle$$
such that $In_{v_1}(f)=f$, $In_{v_2}(f)=z^2-xy^2$, $In_{v_3}(f)=-xy^2-x^{n-1}$, $In_{v_4}(f)=z^2-x^{n-1}$, $In_{v_5}(f)=-xy^2$, $In_{v_6}(f)=-x^{n-1}$ and $In_{v_7}(f)=z^2$. The Gröbner fan $G(X)$ of $X$ is the union $\cup_{i=1}^{7}C_{v_i}$.

\begin{figure}[H]
    \centering
    \includegraphics[width=0.9\linewidth]{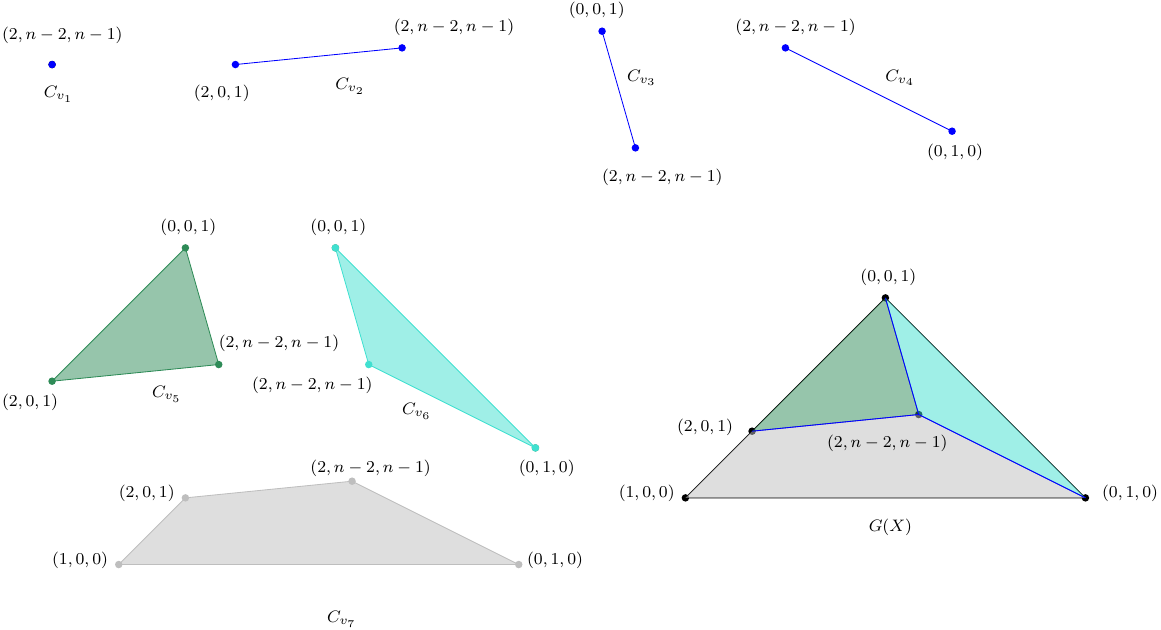}
    \caption{The Gröbner fan and its Gröbner cones for $D_n$-singularities}
\end{figure}

\begin{rem}
The Gröbner fan $G(X)$ of $X$ is same as its dual Newton polyhedron \cite{hc}.
\end{rem}

\noindent For each maximal dimensional cones ($C_{v_{5}}$, $C_{v_{6}}$ and $C_{v_{7}}$) of $G(X)$, the profiles $p_{C_{v_{5}}}$, $p_{C_{v_{6}}}$ and  $p_{C_{v_{7}}}$ are given as 
\begin{figure}[H]
    \centering
    \includegraphics[width=0.9\linewidth]{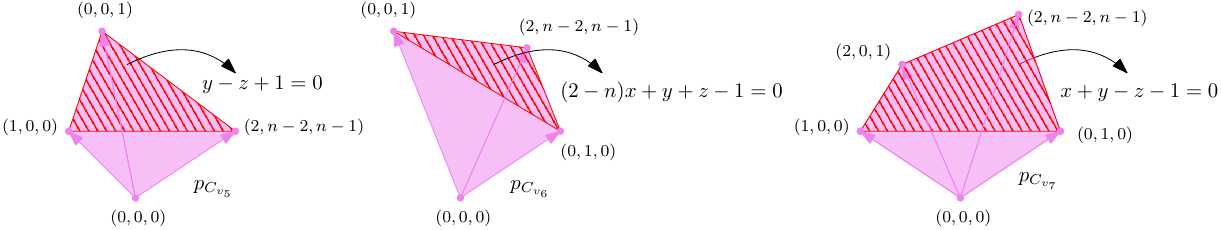}
    \caption{The profiles of maximal dimensional Gröbner cones for $D_n$-singularities}
\end{figure}

\noindent The boundaries of $p_{C_{v_{5}}}$, $p_{C_{v_{6}}}$ and $p_{C_{v_{7}}}$ are $H_1: y-z+1=0$, $H_2: (2-n)x+y+z-1=0$ and $H_3: x+y-z-1=0$ respectively. Note that although $C_{v_7}$ is a non-simplicial cone, all its extremal vectors are on a unique hyperplane.

\noindent We have two different sub-cases depending on $n$ is even or odd.

\noindent $\bullet$ $n$ is even:
\begin{thm} \label{2a}
 $i)$ For $D_n$-singularities, $n$ is even, the elements of the set consisting of $\mathbb{Z}^3 \cap p_{C_{v_{i}}}$ give an embedded toric resolution of $X$ where $C_{v_i}$ is a maximal dimensional Gröbner cone in $G(X)$. Moreover these elements are irreducible means they are free over $\mathbb{Z}$.

$ii)$  For $D_n$-singularities, $n$ is even, the elements on the skeleton of $G(X)$ give the minimal resolution graph of the singularities.
\end{thm}
\begin{proof}
$i)$ The cones $C_{v_5}$, $C_{v_6}$ and $C_{v_7}$ are the maximal dimensional cones for $G(X)$. The elements for $\mathbb{Z}^3 \cap p_{C_{v_{5}}}$ are 

$\bullet$ $(1,0,1), (1,1,2), \ldots,(1,\frac{n-2}{2},\frac{n}{2})$

$\bullet$ $(2,0,1), (2,1,2), \ldots,(2,n-2,n-1)$

$\bullet$ $(0,0,1)$

The elements for $\mathbb{Z}^3 \cap p_{C_{v_{6}}}$ are

$\bullet$ $(0,0,1), (0,1,0),(2,n-2,n-1),(1,\frac{n-2}{2},\frac{n}{2})$

$\bullet$ $(1,0,0),(0,1,0)$

 The elements for $\mathbb{Z}^3 \cap p_{C_{v_{7}}}$ are 

$\bullet$ $(2,0,1), (2,1,2), \ldots,(2,n-2,n-1)$

$\bullet$ $(1,1,1), (1,2,2), \ldots,(1,\frac{n}{2}-\frac{n}{2}-1)$

$\bullet$ $(1,0,0),(0,1,0)$

\noindent A regular refinement of the Gröbner fan $G(X)$ of $X$ is

\begin{figure}[H]
    \centering
    \includegraphics[width=0.7\linewidth]{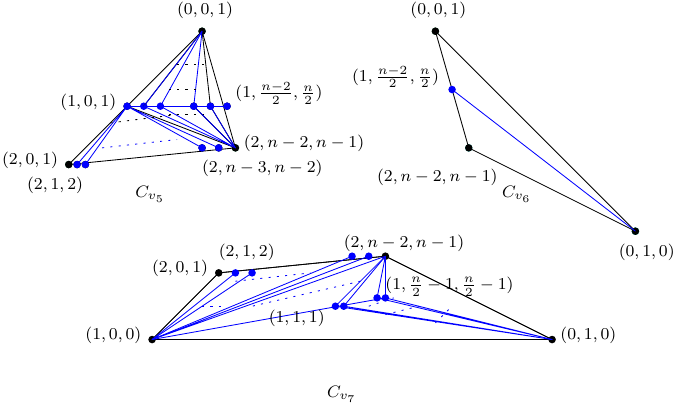}
    \caption{A regular refinement for $D_n$-singularities, n is even}
\end{figure}

\noindent Each element in the refinement is irreducible since it cannot be written as a sum of two other elements.

$ii)$ The elements on the skeleton of $G(X)$ gives 
\begin{figure}[H]
    \centering
    \includegraphics[width=0.7\linewidth]{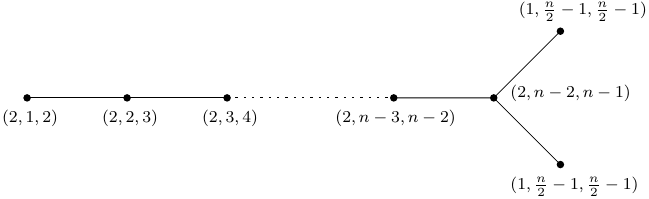}
    \caption{The resolution graph of $D_n$-singularities, n is even}
\end{figure}

\noindent The self-intersection of each vertex is $-2$ so it has no $-1$ curves. Hence it is minimal resolution graph.
\end{proof}

\begin{cor}
 For $D_n$-singularities, $n$ is even, the set of the vectors in the intersection of profile of the maximal dimensional Gröbner cones of $G(X)$ and $\mathbb{Z}^3$ are exactly same as the set of embedded valuations set of $X$ given in \cite{hussein}.  
\end{cor}

\begin{cor}
 The embedded toric resolution that we construct is minimal.   
\end{cor}

\begin{proof}
By Theorem \ref{2a}($i$), the elements give an embedded toric resolution and these elements are irreducible. By Theorem \ref{2a}($ii$), the resolution graph is minimal. Hence the embedded toric resolution we construct is minimal.
\end{proof}

\noindent $\bullet$ $n$ is odd:
\begin{thm} \label{2b}
 $i)$  For $D_n$-singularities, $n$ is odd, the elements of the set consisting of $\mathbb{Z}^3 \cap p_{C_{v_{i}}}$ give an embedded toric resolution of $X$ where $C_{v_i}$ is a maximal dimensional Gröbner cone in $G(X)$. Moreover these elements are irreducible means they are free over $\mathbb{Z}$.

$ii)$  For $D_n$-singularities, $n$ is odd, the elements on the skeleton of $G(X)$ give the minimal resolution graph of the singularities.
\end{thm}
\begin{proof}
$i)$ The cones $C_{v_5}$, $C_{v_6}$ and $C_{v_7}$ are the maximal dimensional cones in $G(X)$. The elements for $\mathbb{Z}^3 \cap p_{C_{v_{5}}}$ are 

$\bullet$ $(1,1,2), (1,2,3), \ldots,(1,\frac{n-3}{2},\frac{n-1}{2})$

$\bullet$ $(2,0,1), (2,1,2), \ldots,(2,n-2,n-1)$

$\bullet$ $(0,0,1)$

The elements for $\mathbb{Z}^3 \cap p_{C_{v_{6}}}$ are

$\bullet$ $(0,0,1), (0,1,0),(2,n-2,n-1),(1,\frac{n-1}{2},\frac{n-1}{2})$

The elements for $\mathbb{Z}^3 \cap p_{C_{v_{7}}}$ are 

$\bullet$ $(2,0,1), (2,1,2), \ldots,(2,n-2,n-1)$

$\bullet$ $(1,1,1), (1,2,2), \ldots,(1,\frac{n-1}{2},\frac{n-1}{2})$

$\bullet$ $(1,0,0),(0,1,0)$

\noindent  A regular refinement of the Gröbner fan $G(X)$ of $X$ is
\begin{figure}[H]
    \centering
    \includegraphics[width=0.7\linewidth]{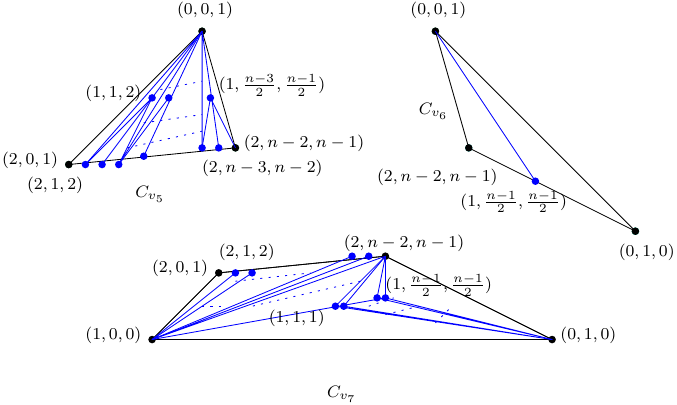}
    \caption{A regular refinement for $D_n$-singularities, n is odd}
\end{figure}

\noindent Each element in the refinement is irreducible since it cannot be written as a sum of two other elements.

$ii)$ The elements on the skeleton of $G(X)$ gives 
\begin{figure}[H]
    \centering
    \includegraphics[width=0.7\linewidth]{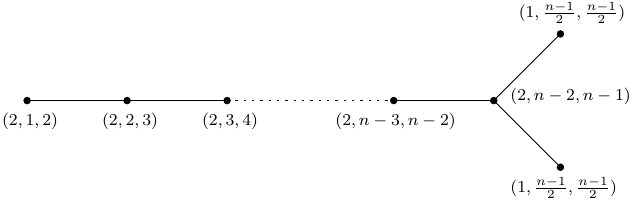}
    \caption{The resolution graph of $D_n$-singularities, n is odd}
\end{figure}
\noindent The self-intersection of each vertex is $-2$ so it has no $-1$ curves. Hence it is minimal resolution graph.
\end{proof}

\begin{cor}
 For $D_n$-singularities, $n$ is odd, the set of the vectors in the intersection of profile of the maximal dimensional Gröbner cones of $G(X)$ and $\mathbb{Z}^3$ are exactly the same as the set of embedded valuations of $X$ given in \cite{hussein}.  
\end{cor}

\begin{cor}
 The embedded toric resolution that we construct is minimal.   
\end{cor}

\begin{proof}
By Theorem \ref{2b}($i$), the elements give an embedded toric resolution and these elements are irreducible. By Theorem \ref{2b}($ii$), the resolution graph is minimal. Hence the embedded toric resolution we construct is minimal.
\end{proof}

\begin{thm}
\cite{hussein} For $D_n$-singularities, when $m\geq 2n-3$, in the $m$-jet scheme over the singular locus there are $n$ irreducible components which is equal to the number of vertices of minimal resolution graph.
\end{thm}
 \noindent In \cite{koreda}, Y. Koreeda gave the construction for $D_4$-singularity. For $n > 4$, showing the construction is still an open question.

\subsection{$E_6$-singularity}
Consider the hypersurface $X=V(f)\subset \mathbb{C}^3$ where $$f(x,y,z)=z^2+y^3+x^4$$ The Gröbner cones are $$C_{v_1}(f)=\langle (3,4,6) \rangle$$
$$C_{v_2}(f)=\langle (1,0,0),(3,4,6) \rangle$$
$$C_{v_3}(f)=\langle (0,0,1),(3,4,6) \rangle$$
$$C_{v_4}(f)=\langle (0,1,0),(3,4,6) \rangle$$
$$C_{v_5}(f)=\langle (1,0,0),(0,0,1),(3,4,6) \rangle$$
$$C_{v_6}(f)=\langle (0,1,0),(0,0,1),(3,4,6) \rangle$$
$$C_{v_7}(f)=\langle (1,0,0),(0,1,0),(3,4,6) \rangle$$ such that $In_{v_1}(f)=f$, $In_{v_2}(f)=z^2+y^3$, $In_{v_3}(f)=y^3+x^4$, $In_{v_4}(f)=z^2+x^4$, $In_{v_5}(f)=y^3$, $In_{v_6}(f)=x^4$ and $In_{v_7}(f)=z^2$. The Gröbner fan $G(X)$ of $X$ is the union $\cup_{i=1}^{7}C_{v_i}$.

\begin{figure}[H]
    \centering
    \includegraphics[width=0.9\linewidth]{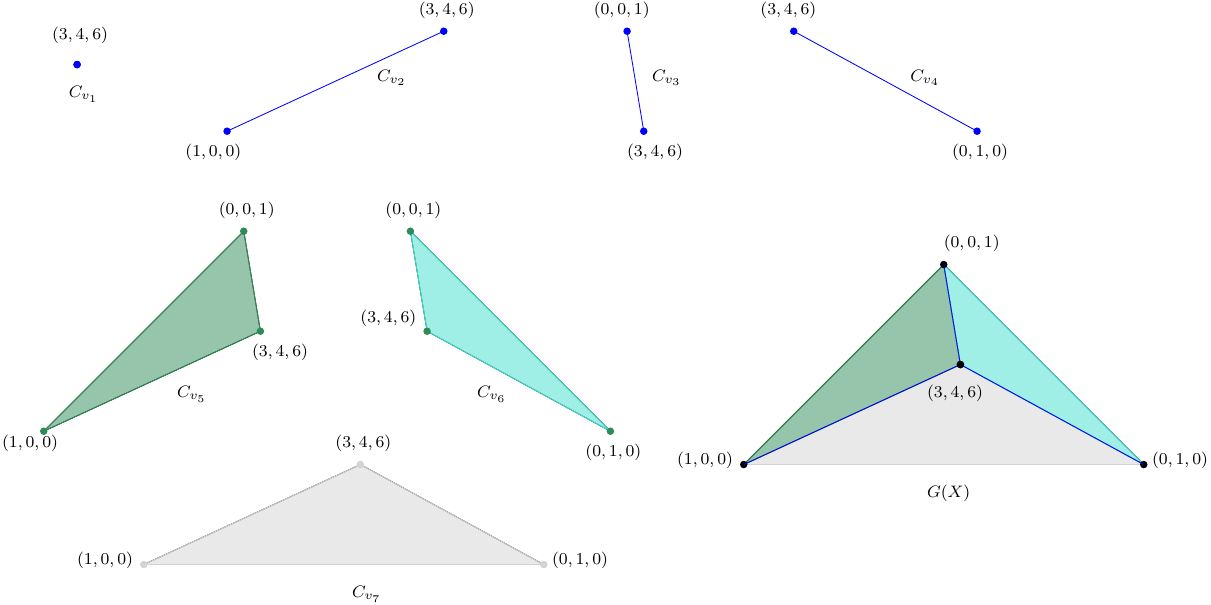}
    \caption{The Gröbner fan and its Gröbner cones for $E_6$-singularity}
\end{figure}

\begin{rem}
The Gröbner fan $G(X)$ of $X$ is same as its dual Newton polyhedron \cite{hc}.
\end{rem} 
\noindent For each maximal dimensional cones ($C_{v_{4}}$, $C_{v_{5}}$ and $C_{v_{6}}$) of $G(X)$, the profiles $p_{C_{v_{4}}}$, $p_{C_{v_{5}}}$ and  $p_{C_{v_{6}}}$ are given as 

\begin{figure}[H]
    \centering
    \includegraphics[width=0.9\linewidth]{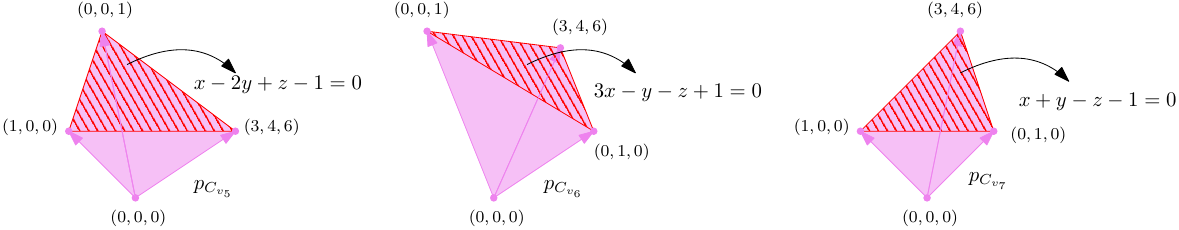}
    \caption{The profiles of maximal dimensional Gröbner cones for $E_6$-singularity}
\end{figure}

\noindent The boundaries of $p_{C_{v_{5}}}$, $p_{C_{v_{6}}}$ and  $p_{C_{v_{7}}}$ are $H_1:x-2y+z-1=0$, $H_2:3x-y-z+1=0$ and $H_3:x+y-z-1=0$ respectively.

\begin{thm} \label{3}
$i)$ For $E_6$-singularity, the elements of the set consisting of $\mathbb{Z}^3 \cap p_{C_{v_{i}}}$ give an embedded toric resolution of $X$ where $C_{v_i}$ is a maximal dimensional Gröbner cone in $G(X)$. Moreover these elements are irreducible means they are free over $\mathbb{Z}$.

$ii)$ For $E_6$-singularity, the elements on the skeleton of $G(X)$ give minimal resolution graph of the singularity.
\end{thm}

\begin{proof}
$i)$ The cones $C_{v_5}$, $C_{v_6}$ and $C_{v_7}$ are the maximal dimensional cones in $G(X)$. The elements for $\mathbb{Z}^3 \cap p_{C_{v_{5}}}$ are 

$\bullet$ $(1,0,0),(0,0,1),(3,4,6),(1,1,2),(2,2,3)$

The elements for $\mathbb{Z}^3 \cap p_{C_{v_{6}}}$ are

$\bullet$ $(0,1,0),(0,0,1),(3,4,6),(1,1,2),(2,3,4)$

The elements for $\mathbb{Z}^3 \cap p_{C_{v_{7}}}$ are

$\bullet$ $(1,0,0),(0,1,0),(3,4,6),(1,1,1),(1,2,2),(2,2,3),(2,3,4)$

\noindent A regular refinement of the Gröbner fan $G(X)$ of $X$ is
\begin{figure}[H]
    \centering
    \includegraphics[width=0.6\linewidth]{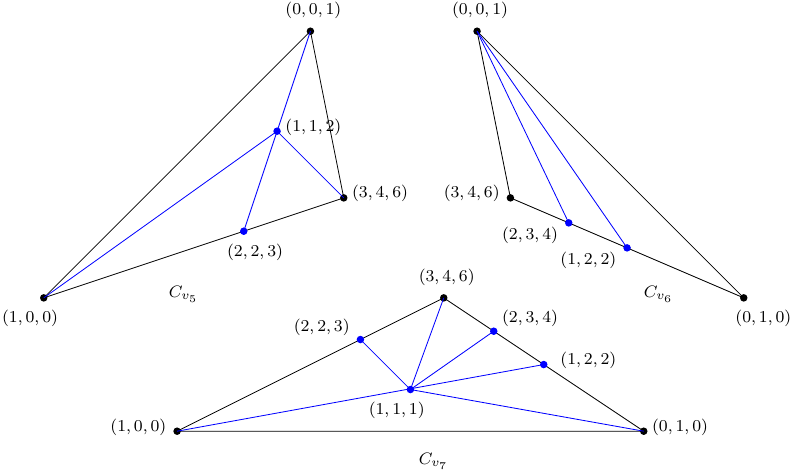}
    \caption{A regular refinement for $E_6$-singularity}
\end{figure}

\noindent Each element in the refinement is irreducible since it cannot be written as a sum of two other elements.

$ii)$ The elements on the skeleton of $G(X)$ gives 
\begin{figure}[H]
    \centering
    \includegraphics[width=0.6\linewidth]{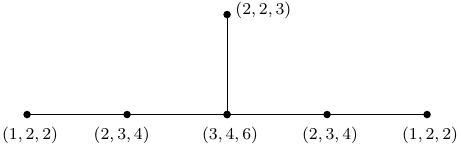}
    \caption{The resolution graph of $E_6$-singularity}
\end{figure}
\noindent The self-intersection of each vertex is $-2$ so it has no $-1$ curves. Hence it is minimal resolution graph.
\end{proof}

\begin{cor}
For $E_6$-singularity, the set of the vectors in the intersection of profile of the maximal dimensional Gröbner cones of $G(X)$ and $\mathbb{Z}^3$ are exactly same as the set of embedded valuations of $X$ given in \cite{hussein}. 
\end{cor}

\begin{cor}
 The embedded toric resolution that we construct is minimal.   
\end{cor}

\begin{proof}
By Theorem \ref{3}($i$), the elements give an embedded toric resolution and these elements are irreducible. By Theorem \ref{3}($ii$), the resolution graph are minimal. Hence the embedded toric resolution is minimal.  
\end{proof}

\begin{thm}
\cite{hussein} For $E_6$-singularity , when $m \geq 11$, in the $m$-jet scheme over the singular locus there are $six$ irreducible components which is equal to the number of vertices of its minimal resolution graph.
\end{thm}
\subsection{$E_7$-singularity}
Consider the hypersurface $X=V(f) \subset \mathbb{C}^3$ where $$f(x,y,z)=x^2+y^3+yz^3$$ The Gröbner cones are $$C_{v_1}(f)=\langle (9,6,4) \rangle$$
$$C_{v_2}(f)=\langle (1,0,0),(9,6,4) \rangle$$
$$C_{v_3}(f)=\langle (0,0,1),(9,6,4) \rangle$$
$$C_{v_4}(f)=\langle (1,2,0),(9,6,4) \rangle$$
$$C_{v_5}(f)=\langle (1,0,0),(0,0,1),(9,6,4) \rangle$$
$$C_{v_6}(f)=\langle (0,1,0),(0,0,1),(1,2,0),(9,6,4) \rangle$$
$$C_{v_7}(f)=\langle (1,0,0),(1,2,0),(9,6,4) \rangle$$ suct that $In_{v_1}(f)=f$, $In_{v_2}(f)=y^3+yz^3$, $In_{v_3}(f)=x^2+y^3$, $In_{v_4}(f)=x^2+yz^3$, $In_{v_5}(f)=y^3$, $In_{v_6}(f)=x^2$ and $In_{v_7}(f)=yz^3$. The Gröbner fan $G(X)$ of $X$ is the union $\cup_{i=1}^{7}C_{v_i}$.

\begin{figure}[H]
    \centering
    \includegraphics[width=0.9\linewidth]{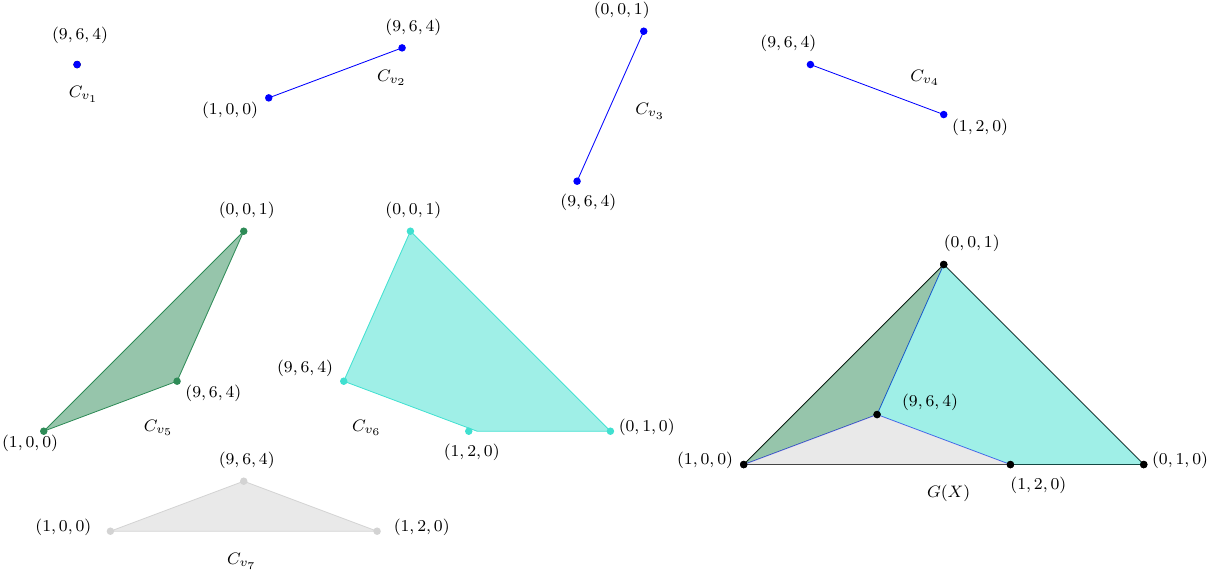}
    \caption{The Gröbner fan and its Gröbner cones for $E_7$-singularity}
\end{figure}

\begin{rem}
 The Gröbner fan $G(X)$ of $X$ is same as its dual Newton polyhedron \cite{hc}.   
\end{rem}
\noindent For each maximal dimensional cones ($C_{v_{5}}$, $C_{v_{6}}$ and $C_{v_{7}}$) of $G(X)$, the profiles $p_{C_{v_{5}}}$, $p_{C_{v_{6}}}$ and  $p_{C_{v_{7}}}$ are given as 
\begin{figure}[H]
    \centering
    \includegraphics[width=0.9\linewidth]{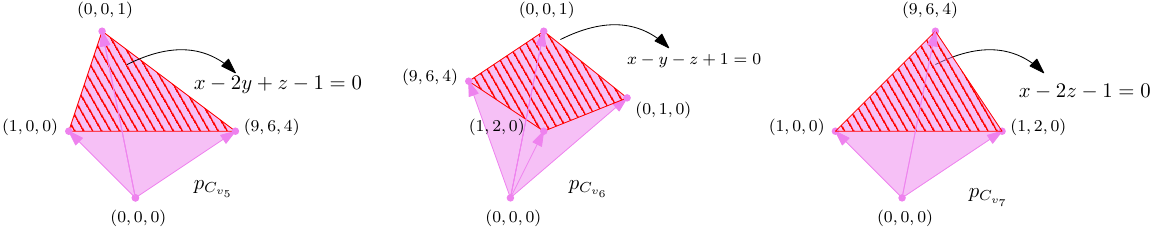}
    \caption{The profiles of maximal dimensional Gröbner cones for $E_7$-singularity}
\end{figure}

\noindent The boundaries of $p_{C_{v_{5}}}$, $p_{C_{v_{6}}}$ and  $p_{C_{v_{7}}}$ are $H_1: x-2y+z-1=0$, $H_2: x-y-z+1=0$ and $H_3: x-2z-1=0$ respectively. Note that although $C_{v_6}$ is a non-simplicial cone, all its extremal vectors are on a unique hyperplane.

\begin{thm}\label{4}
$i)$ For $E_7$-singularity, the elements of the set consisting of $\mathbb{Z}^3 \cap p_{C_{v_{i}}}$ give an embedded toric resolution of $X$where $C_{v_i}$ is a maximal dimensional Gröbner cone in $G(X)$. Moreover these elements are irreducible means they are free over $\mathbb{Z}$.

$ii)$ For $E_7$-singularity, the elements on the skeleton of $G(X)$ give minimal resolution graph of the singularity.
\end{thm}

\begin{proof}
$i)$ The cones $C_{v_5}$, $C_{v_6}$ and $C_{v_7}$ are the maximal dimensional cones in $G(X)$.  The elements for $\mathbb{Z}^3 \cap p_{C_{v_{5}}}$ are 

$\bullet$ $(1,0,0),(0,0,1),(9,6,4),(6,4,3),(5,3,2),(3,2,2),(2,1,1)$

The elements for $\mathbb{Z}^3 \cap p_{C_{v_{6}}}$ are

$\bullet$ $(0,1,0),(0,0,1),(9,6,4),(1,2,0),(7,5,3),(5,4,2),(3,3,1),(6,4,3),(3,2,2)$

$\bullet$ $(4,3,2),(2,2,1),(1,1,1)$

The elements for $\mathbb{Z}^3 \cap p_{C_{v_{7}}}$ are

$\bullet$ $(1,0,0),(1,2,0),(9,6,4),(1,1,1),(3,2,1),(5,3,2),(7,5,3),(5,4,2),(3,3,1)$

\noindent A regular refinement of the Gröbner fan $G(X)$ of $X$ is
\begin{figure}[H]
    \centering
    \includegraphics[width=0.6\linewidth]{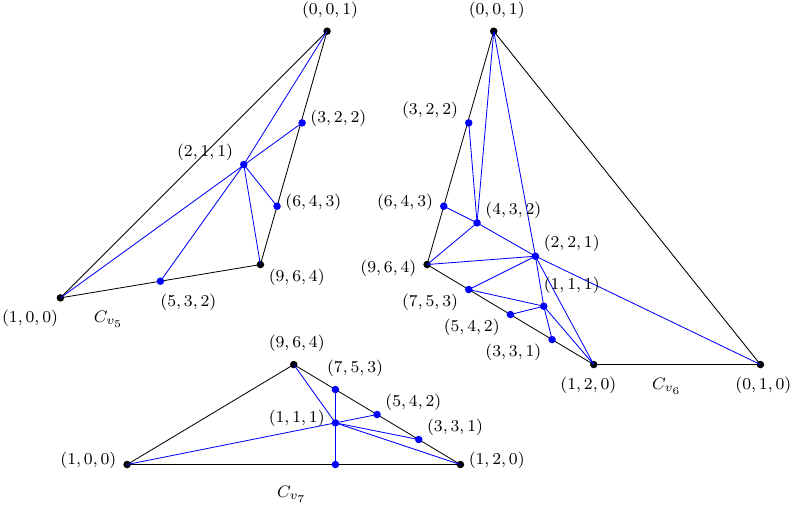}
    \caption{A regular refinement for $E_7$-singularity}
\end{figure}

\noindent Each element is irreducible since it cannot be written as a sum of two other elements.

$ii)$ The elements on the skeleton of $G(X)$ gives 
\begin{figure}[H]
    \centering
    \includegraphics[width=0.6\linewidth]{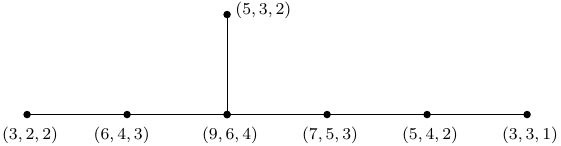}
    \caption{The resolution graph of $E_7$-singularity}
\end{figure}

\noindent The self-intersection of each vertex is $-2$ so it has no $-1$ curves. Hence it is minimal resolution graph.
\end{proof}

\begin{cor}
For $E_7$-singularity, the set of the vectors in the intersection of profile of the maximal dimensional Gröbner cones of $G(X)$ and $\mathbb{Z}^3$ are exactly same as the set of embedded valuations of $X$ given in \cite{hussein}. 
\end{cor}

\begin{cor}
 The embedded toric resolution that we construct is minimal.   
\end{cor}

\begin{proof}
By Theorem \ref{4}($i$), the elements give an embedded toric resolution and these elements are irreducible. By Theorem \ref{4}($ii$), the resolution graph is minimal. Hence the embedded toric resolution is minimal.
\end{proof}

\begin{thm}
\cite{hussein} For $E_7$-singularity , when $m \geq 17$, in the $m$-jet scheme over the singular locus there are $seven$ irreducible components which is equal to the number of vertices of its minimal resolution graph.
\end{thm}

\subsection{$E_8$-singularity}
Consider the hypersurface $X=V(f)\subset \mathbb{C}^3$ where $$f(x,y,z)=z^2+y^3+x^5$$ The Gröbner cones are $$C_{v_1}(f)=\langle (6,10,15) \rangle$$
$$C_{v_2}(f)=\langle (1,0,0),(6,10,15) \rangle$$
$$C_{v_3}(f)=\langle (0,0,1),(6,10,15) \rangle$$
$$C_{v_4}(f)=\langle (0,1,0),(6,10,15) \rangle$$
$$C_{v_5}(f)=\langle (1,0,0),(0,0,1),(6,10,15) \rangle$$
$$C_{v_6}(f)=\langle (0,1,0),(0,0,1),(6,10,15) \rangle$$
$$C_{v_7}(f)=\langle (1,0,0),(0,1,0),(6,10,15) \rangle$$ such that $In_{v_1}(f)=f$, $In_{v_2}(f)=z^2+y^3$, $In_{v_3}(f)=y^3+x^5$, $In_{v_4}(f)=z^2+x^5$, $In_{v_5}(f)=y^3$, $In_{v_6}(f)=x^5$ and $In_{v_7}(f)=z^2$. The Gröbner fan $G(X)$ of $X$ is the union $\cup_{i=1}^{7}C_{v_i}$.

\begin{figure}[H]
    \centering
    \includegraphics[width=0.9\linewidth]{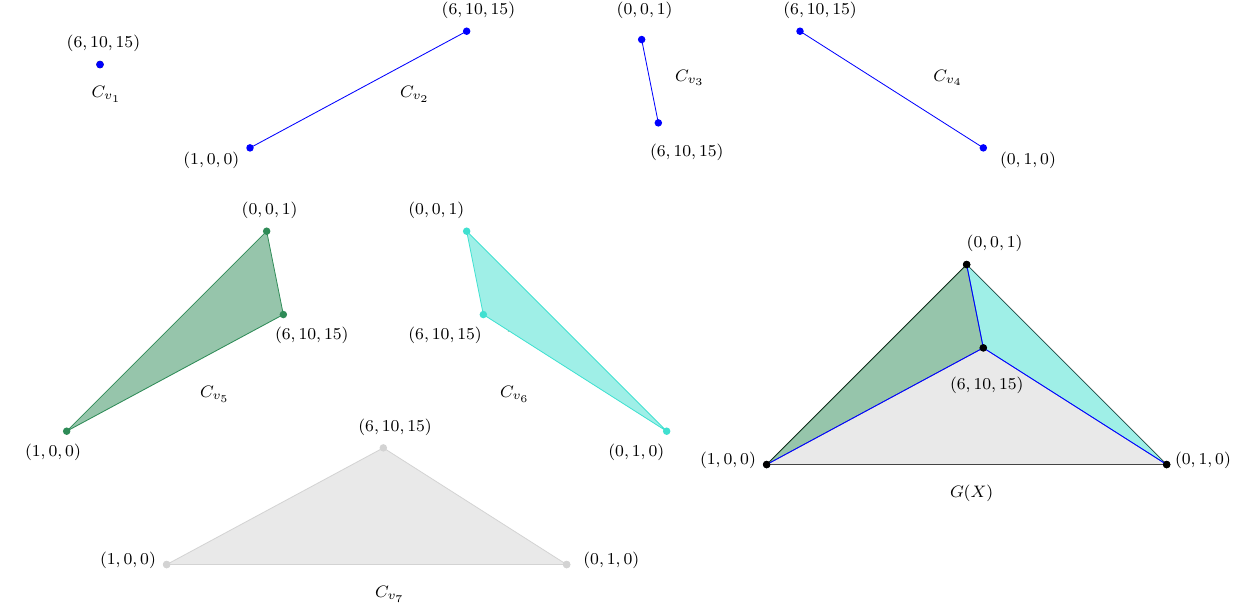}
    \caption{The Gröbner fan and its Gröbner cones for $E_8$-singularity}
\end{figure}

\begin{rem}
The Gröbner fan $G(X)$ of $X$ is same as its dual Newton polyhedron \cite{hc}.    
\end{rem}
\noindent For each maximal dimensional cones ($C_{v_{5}}$, $C_{v_{6}}$ and $C_{v_{7}}$) of $G(X)$, the profiles $p_{C_{v_{5}}}$, $p_{C_{v_{6}}}$ and  $p_{C_{v_{7}}}$ are given as 
\begin{center}
\includegraphics[scale=0.6]{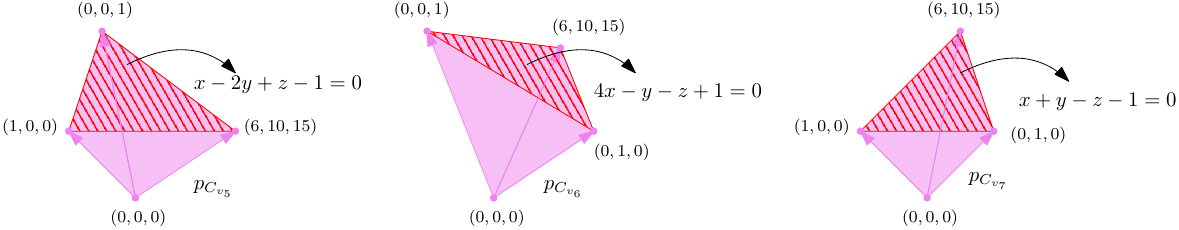}
\end{center}
The boundaries of $p_{C_{v_{5}}}$, $p_{C_{v_{6}}}$ and  $p_{C_{v_{7}}}$ are $H_1: x-2y+z-1=0$, $H_2: 4x-y-z+1=0$ and $H_3: x+y-z-1=0$ respectively.

\begin{thm}\label{5}
$i)$ For $E_8$-singularity, the elements of the set consisting of $\mathbb{Z}^3 \cap p_{C_{v_{i}}}$ give an embedded toric resolution of $X$ where $C_{v_i}$ is a maximal dimensional Gröbner cone in $G(X)$. Moreover these elements are irreducible means they are free over $\mathbb{Z}$.

$ii)$ For $E_8$-singularity, the elements on the skeleton of $G(X)$ give minimal resolution graph of the singularity.
\end{thm}

\begin{proof}
$i)$ The cones $C_{v_5}$, $C_{v_6}$ and $C_{v_7}$ are the maximal dimensional cones in $G(X)$. The elements for $\mathbb{Z}^3 \cap p_{C_{v_{5}}}$ are 

$\bullet$ $(0,1,0),(0,0,1),(6,10,15),(2,2,3),(3,4,6),(4,6,9),(5,8,12),(1,1,2)$

$\bullet$ $(2,3,5),(3,5,8)$

The elements for $\mathbb{Z}^3 \cap p_{C_{v_{6}}}$ are

$\bullet$ $(0,1,0),(0,0,1),(6,10,15),(2,4,5),(4,7,10),(3,5,8),(1,2,3)$

The elements for $\mathbb{Z}^3 \cap p_{C_{v_{7}}}$ are

$\bullet$ $(1,0,0),(0,1,0),(6,10,15),(2,2,3),(3,4,6),(4,6,9),(5,8,12),(2,4,5),(4,7,10)$

$\bullet$ $(1,2,2),(2,3,4),(3,4,6),(3,5,7),(1,2,2),(2,3,4)$

\noindent A regular refinement of the Gröbner fan $G(X)$ of $X$ is
\begin{figure}[H]
    \centering
    \includegraphics[width=0.6\linewidth]{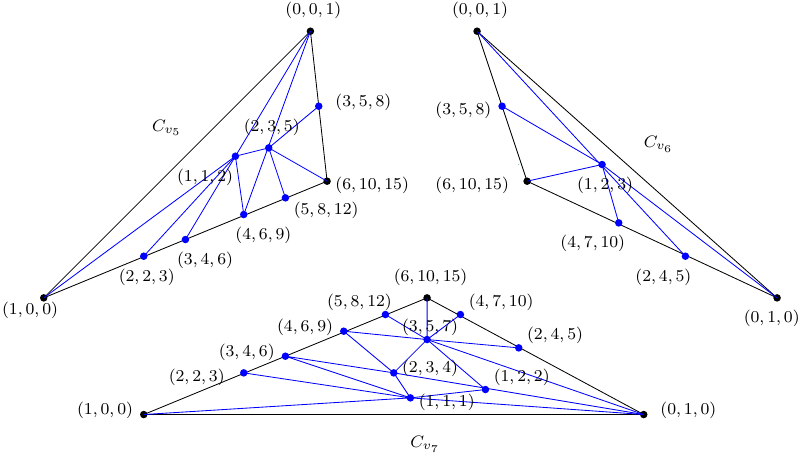}
    \caption{A regular refinement for $E_8$-singularity}
\end{figure}

\noindent Each element in the refinement is irreducible since it cannot be written as a sum of two other elements.

$ii)$ The elements on the skeleton of $G(X)$ gives 
\begin{figure}[H]
    \centering
    \includegraphics[width=0.6\linewidth]{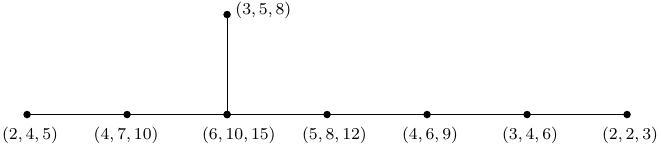}
    \caption{The resolution graph of $E_8$-singularity}
\end{figure}
\noindent The self-intersection of each vertex is $-2$ so it has no $-1$ curves. Hence it is minimal resolution graph.
\end{proof}

\begin{cor}
For $E_8$-singularity, the set of the vectors in the intersection of profile of the maximal dimensional Gröbner cones of $G(X)$ and $\mathbb{Z}^3$ are exactly same as the set of embedded valuations set of $X$ given in \cite{hussein}.
\end{cor}
\begin{cor}
 The embedded toric resolution that we construct is minimal.   
\end{cor}

\begin{proof}
By Theorem \ref{5}($i$), the elements give a resolution and these elements are irreducible. By Theorem \ref{5}($ii$), the resolution graph is minimal. Hence the embedded toric resolution is minimal.
\end{proof}

\begin{thm}
\cite{hussein} For $E_8$-singularity , when $m \geq 29$, in the $m$-jet scheme over the singular locus there are $eight$ irreducible components which is equal to the number of vertices of its minimal resolution graph.
\end{thm}


\vskip.4cm

\noindent B\"{u}\c{s}ra KARADEN\.{I}Z \c SEN \\
e-mail: busrakaradeniz@gtu.edu.tr\\
Gebze Technical University \\ 
Department of Mathematics  \\
41400, Kocaeli, Turkey
  \\



\end{document}